\documentclass[12pt,twoside]{amsart}

\usepackage{amssymb, amsmath,mathrsfs,color}
\usepackage[OT2,T1]{fontenc}

\allowdisplaybreaks

\def\ls{\lesssim}

\def\fz{\infty}

\def\az{\alpha}
\def\ez{{\epsilon}}

\def\nn{\mathbb{N}}

\def\fz{\infty}

\def\r{\right}
\def\lf{\left}

\def\rlz{\R\times\R}

\def\mbhh{{[ b, H_1H_2]}}
\def\dyz{{\,dy_1\,dy_2}}

\newcommand{\R}{\mathbb{R}}
\newcommand{\RR}{{\mathbb{R}\times \mathbb{R}}}

\def\ltp{{L^2\,(\R^2)}}
\def\hop{{h^1(\RR)}}
\def\bmop{{{\rm bmo}\,(\RR)}}

\newtheorem{theorem}{Theorem}[section]
\newtheorem{lemma}[theorem]{Lemma}

\newtheorem{coro}[theorem]{Corollary}

\theoremstyle{definition}
\newtheorem{definition}[theorem]{Definition}
\newtheorem{remark}[theorem]{Remark}

\numberwithin{equation}{section}

\DeclareSymbolFont{cyrletters}{OT2}{wncyr}{m}{n}
\DeclareMathSymbol{\Sha}{\mathalpha}{cyrletters}{"58}

\allowdisplaybreaks

\begin{document}

\title[Commutators, Little BMO and Weak Factorization]{Commutators, Little BMO and Weak Factorization}

\date{\today}


\author[X. T. Duong]{Xuan Thinh Duong$^\dagger$}
\address{Xuan Thinh Duong, Department of Mathematics\\
Macquarie University\\
NSW, 2109, Australia}
\email{xuan.duong@mq.edu.au}
\thanks{$\dagger$ Research supported in part by ARC DP 160100153.}

\author[J. Li]{Ji Li$^\ddagger$}
\address{Ji Li, Department of Mathematics\\
Macquarie University\\
NSW, 2109, Australia}
\email{ji.li@mq.edu.au}
\thanks{$\ddagger$ Research supported by ARC DP 160100153 and MQ New Staff Grant}

\author[B. D. Wick]{Brett D. Wick$^\ast$}
\address{Brett D. Wick, Department of Mathematics\\
Washington University -- St. Louis\\
One Brookings Drive\\
St. Louis, MO USA 63130}
\email{wick@math.wustl.edu}
\thanks{$\ast$ Research supported in part by National Science Foundation
DMS grants \# 1560955 and \#1603246.}

\author[D. Yang]{Dongyong Yang$^\star$}
\address{Dongyong Yang, School of Mathematical Sciences\\
Xiamen University \\
Xiamen 361005,  China}
\email{dyyang@xmu.edu.cn}
\thanks{$\star$ Research supported by
 the NNSF of China (Grant No. 11571289) and the State Scholarship Fund of China (No. 201406315078)}



\begin{abstract}
In this paper, we provide a direct and constructive proof of weak factorization of $h^1(\RR)$ (the predual of little BMO space bmo$(\mathbb{R}\times\mathbb{R})$ studied by Cotlar-Sadosky and Ferguson-Sadosky),
i.e., for every $f\in h^1(\mathbb{R}\times\mathbb{R})$ there exist sequences $\{\alpha_j^k\}\in\ell^1$ and functions
$g_j^k,h^k_j\in L^2(\mathbb{R}^2)$ such that
\begin{align*}
f=\sum_{k=1}^\infty\sum_{j=1}^\infty\alpha^k_j\Big(\, h^k_j H_1H_2 g^k_j  - g^k_j H_1H_2 h^k_j\Big)
\end{align*}
in the sense of $h^1(\RR)$, where $H_1$ and $H_2$ are the Hilbert transforms on the first and second variable, respectively. Moreover,
the norm $\|f\|_\hop$ is given in terms of $\|g^k_j\|_{L^2(\mathbb{R}^2)}$ and $\|h^k_j\|_{L^2(\mathbb{R}^2)}$.
By duality, this directly implies a lower bound on the norm of the commutator $[b,H_1H_2]$ in terms of $\|b\|_{{\rm bmo}(\mathbb{R}\times\mathbb{R})}$.

Our method bypasses the use of analyticity and the Fourier transform, and hence can be extended to the higher dimension case in an arbitrary $n$-parameter setting for the Riesz transforms.
\end{abstract}

\maketitle

\section{Introduction and Statement of Main Results}

As motivation for this paper we point to two fundamental results in complex analysis and harmonic analysis.  An important result, obtained by Coifman, Rochberg, and Weiss in \cite{crw} shows that for the Hilbert transform $ Hf(x)=\int_{\R} \frac{f(y)}{y-x}dy$ and the commutator between multiplication by $b$ (i.\,e., $M_bf=bf$) and the Hilbert transform, $[b,H]:=M_bH-H M_b$, that:
$$
\left\Vert [b,H]:L^2(\mathbb{R})\to L^2(\mathbb{R})\right\Vert \approx \sup_{Q}\Big( \frac{1}{\left\vert Q\right\vert}\int_Q\Big\vert b(x)-\frac{1}{\vert Q\vert}\int_{Q} b(y)dy\Big\vert^2 dx\Big)^{\frac{1}{2}},
$$
where the supremum is taken over intervals $Q$ in $\R$ and the right-hand side is the well-known BMO$(\R)$ norm.  To obtain this, they used methods of harmonic analysis that were general enough to work for certain Calder\'on--Zygmund operators, and in particular the Riesz transforms: $R_jf(x):=c_n \int_{\mathbb{R}^n} f(y) \frac{x_j-y_j}{\left\vert x-y\right\vert^{n+1}}dy$, $1\leq j\leq n$, and obtained:
$$
\max_{1\leq j\leq n}\left\Vert [b,R_j]:L^2(\mathbb{R}^n)\to L^2(\mathbb{R}^n)\right\Vert \approx \sup_{Q}\Big( \frac{1}{\left\vert Q\right\vert}\int_Q\Big\vert b(x)-\frac{1}{\left\vert Q\right\vert}\int_Q b(y)dy\Big\vert^2 dx\Big)^{\frac{1}{2}},
$$
where the supremum is taken over cubes $Q$ in $\R^n$ and the right-hand side is the well-known BMO$(\R^n)$ norm.  Commutators play an important role in harmonic analysis, complex analysis, and partial differential equations (see for example \cite{ais,clms,kpv}) and have a characterization of their boundedness in terms of the symbol $b$ is extremely useful.

Nehari studied and characterized the boundedness of Hankel operators in \cite{n}.  Recall that $H^2(\mathbb{R}^2_+)$ is the space of functions that are analytic on the upper half-plane and have boundary values belonging to $L^2(\mathbb{R})$.  Let $\mathbb{P}_+:L^2(\mathbb{R})\to H^2(\mathbb{R}^2_+)$ denote the orthogonal projection between these spaces and so we have that $L^2(\mathbb{R})=H^2_+(\mathbb{R}^2_+)\oplus H^2_-(\mathbb{R}^2_+)$ where $H^2_{\pm}(\mathbb{R}^2_+)$ is supported on the positive/negative Fourier frequencies.  Then define the Hankel operator $h_b(f):=\mathbb{P}_-(bf)$ and Nehari's Theorem, stated in modern terminology, is then the relationship:
$
\left\Vert h_b:H^2(\mathbb{R}^2_+)\to H^2_-(\mathbb{R}^2_+)\right\Vert\approx \left\Vert b\right\Vert_{\textnormal{BMO}(\mathbb{R}^2_+)}.
$

There is a strong connection between the results of \cite{crw} and \cite{n}.  To see this recall that we have $H=i\mathbb{P}_+-i\mathbb{P}_-$ where $\mathbb{P}_+$  and $\mathbb{P}_-$  are the projections onto the positive and negative Fourier supports respectively.  It is then a simple computation to show that: $\displaystyle [b,H]=h_b-h_{\overline{b}}^{\ast}$.
As the domains and ranges of the operators $h_b$ and $h_{\overline{b}}^{\ast}$ are orthogonal, Nehari's Theorem and the characterization of commutators can then easily be deduced from one another.

Via $H^1-$BMO duality and some standard functional analysis it is direct to see that the commutator theorem can be translated to the following statement:  For every $f\in H^1(\mathbb{R})$, the real Hardy space, there exist functions $g_j, h_j\in L^2(\mathbb{R)}$ so that $\displaystyle f=\sum_{j=1}^{\infty} g_jHh_j+h_jHg_j$ in the sense of $H^1(\R)$ and
\begin{align*}
\left\Vert f\right\Vert_{H^1(\mathbb{R})}\approx \inf\Big\{\sum_{j=1}^{\infty} \left\Vert g_j\right\Vert_{L^2(\R)} \left\Vert h_j\right\Vert_{L^2(\R)}: f= \sum_{j=1}^{\infty} g_jHh_j+h_jHg_j\Big\},
\end{align*}
where the infimum is taken over all possible representations of $f$ as above (for the definition of $H^1(\mathbb{R})$ see \cite{fs72}).  In fact these factorization results and corresponding commutator results are always equivalent to each other. For more details about the classical Nehari Theorem and background, we refer to the note of Lacey \cite{l} and the references therein.

Extensions of the commutator results and Nehari's Theorem have received lots of attention; in particular we focus on the extensions in the product setting for the little BMO space bmo$(\RR)$, introduced and studied by M. Cotlar and Sadosky \cite{CS2}
in connection with weighted norm inequalities for the product Hilbert transform. For this reason, the space bmo$(\RR)$ was
originally defined in terms of the Hilbert transforms, one for each variable. The characterization of bmo$(\RR)$ in terms of
mean oscillation on rectangles was given later in \cite{CS2}. For our purpose here, we take this characterization of bmo$(\RR)$
as our starting point. Note that in \cite{CS2} and \cite{fs}, they stated the results on bidisc. Here we state the results on $\RR$ and study the real analysis approach.  More precisely, a function $b\in L_{loc}^1(\R^2)$ is in bmo$(\RR)$ if
\begin{align}
\|b\|_{{\rm bmo}(\RR)}:= \sup_{R\subset \RR} \frac1{|R|}\iint_R |b(x_1,x_2)-b_R|dx_1dx_2 <\infty,
\end{align}
where
$$ b_R := \frac1{|R|}\iint_R b(x_1,x_2) dx_1dx_2 $$
is the mean value of $b$ over the rectangle $R$.

It is well known that bmo$(\RR)$ coincides with the space of integrable functions which are uniformly of bounded mean oscillation in each variable separately \cite{CS2}. Moreover, from Ferguson--Sadosky \cite{fs}, we have the following equivalent characterizations for bmo$(\RR)$.
\begin{theorem}[\cite{fs}]
\label{1.1}
Let $b\in L_{loc}^1(\R^2)$. The following conditions are equivalent:
\begin{enumerate}
\item[(i)] $b\in $ bmo$(\RR)$;
\item[(ii)] The commutators $[b,H_1]$ and $[b,H_2]$ are both bounded on $L^2(\R^2)$;
\item[(iii)] The commutator $[b,H_1H_2]$ is bounded on $L^2(\R^2)$.
\end{enumerate}
\end{theorem}
We note that bmo$(\RR)$ can also be equivalently characterized by big Hankel operators and by certain Carleson measures.
For the details, we refer to \cite{fs}.

It was shown in \cite{CS2} that the predual of bmo$(\RR)$ coincides with $H^1_{Re}(\mathbb R)\otimes L^1(\mathbb R) + H^1_{Re}(\mathbb R)\otimes L^1(\mathbb R)$. Based on the result in \cite{CS2}, Ferguson--Sadosky \cite{fs} obtained the weak factorization for
$H^1_{Re}(\mathbb R)\otimes L^1(\mathbb R) + H^1_{Re}(\mathbb R)\otimes L^1(\mathbb R)$.

The aim of this paper is to provide a direct and constructive proof for the weak factorization for predual of bmo$(\RR)$, which
implies the equivalence of (i) and (iii) in Theorem \ref{1.1} directly and our result here bypasses the use of Fourier transform and hence can be extended to the higher dimension case in an arbitrary $n$-parameter setting for the Riesz transforms.
To get this, we note that in \cite{fs}, Ferguson-Sadosky also showed that the predual of  bmo$(\RR)$ can be characterised in terms of rectangular atoms.
\begin{definition}[\cite{fs}]\label{def atom}
An atom on $\RR$ is a function
$a\in L^\infty(\R^2)$ supported on a rectangle $R \subset \RR$ with $\|a\|_\infty \leq |R|^{-1}$ and satisfying the cancellation property
$$  \int_{\R^2} a(x_1,x_2)dx_1dx_2=0. $$
Let $Atom(\RR)$ denote the collection of all such atoms.
\end{definition}
\begin{definition}[\cite{fs}]\label{def little h1}
The atomic Hardy space $h^1(\RR)$ is defined as the set of functions of the form
\begin{align}\label{atomic decomposition1}
f=\sum_i\alpha_i a_i
\end{align}
with $\{a_i\}_i\subset Atom(\RR)$, $\{\alpha_i\}_i \subset \mathbb{C}$ and $\sum_i|\alpha_i| <\infty$.
Moreover, $h^1(\RR)$  is equipped with the norm $\|f\|_{h^1(\RR)}:= \inf \sum_i|\alpha_i|$ where
the infimum is taken over all possible decompositions of $f$ in the form \eqref{atomic decomposition1}.
\end{definition}

Then we have the following result from \cite{fs} on the duality of the atomic Hardy space $h^1$ and little bmo, whose proof will be sketchedt in Section 2 for the convenience of the reader.
\begin{theorem}[\cite{fs}]\label{thm fs} A function $b\in L^1_{loc}(\R^2)$  is in {\rm bmo}$(\RR)$ if and only if
$$ \sup_{a\in Atom(\RR)} \int_{\R^2} b(x_1,x_2)a(x_1,x_2)dx_1dx_2 <\infty. $$
Consequently, the predual of {\rm bmo}$(\RR)$ is $h^1(\RR)$.
\end{theorem}

Our main result of this article is the following.

\begin{theorem}\label{thm: main}
For every $f\in h^1(\RR)$, there exist sequences $\{\alpha_j^k\}_{j}\in\ell^1$ and functions
$g_j^k,h^k_j\in\ltp$ such that
\begin{align}\label{represent of H1}
f=\sum_{k=1}^\fz\sum_{j=1}^\fz\alpha^k_j\, \Pi\lf(g^k_j,h^k_j\r)
\end{align}
in the sense of $h^1(\RR)$, where $\Pi(f,g )$ is the bilinear form defined as
\begin{align}\label{bilinear form}
\Pi(g,h ) :=  h H_1H_2 g - g H_1H_2 h.
\end{align}
Moreover, we have that 
\begin{eqnarray*}
\|f\|_{h^1(\RR)}\approx\inf\Big\{\sum_{k=1}^\fz\sum_{j=1}^\fz\lf|\az^k_j\,\r|\lf\|g^k_j\r\|_\ltp\lf\|h^k_j\r\|_\ltp\Big\},
\end{eqnarray*}
where the infimum is taken over all representations of $f$ in the form \eqref{represent of H1} and the implicit constants
are independent of $f$.
\end{theorem}
\begin{remark}\label{remark main compact}
The functions $g_j^k$ and $h^k_j$ in the main theorem above 
are actually in $L^\infty(\R^2)$ with compact support.
\end{remark}

By duality, we obtain the lower bound of the  commutator $[b,H_1H_2]$, which was known from the work of Ferguson and Sadosky in \cite{fs} (see Theorem \ref{1.1}) .
\begin{coro}\label{thm: lower bound}
Let $b\in L^1(\R^2)$.
 If $\mbhh$ is bounded on $\ltp$, then we get that $b\in\bmop$ and there exists a constant $C$ so that
\begin{eqnarray*}
\|b\|_{\bmop}&\le& C\left\|\mbhh :\ltp\to\ltp\right\Vert.
\end{eqnarray*}
\end{coro}

We further remark that in Theorem \ref{thm: main} and Corollary \ref{thm: lower bound}  it is possible to change $L^2$ to $L^p$ for $1<p<\infty$ and to replace the factorization in terms of $L^p$ and $L^q$, where $\frac{1}{p}+\frac{1}{q}=1$.  We leave these standard modifications to the reader.  Also, as can be seen from the proofs given below, the role of the Hilbert transforms play no substantive role and in fact work for the Riesz transforms just as easily.  In the interest of ease of presentation, we have focused on the proof with the Hilbert transforms and leave the direct modifications again to the reader.

We also point out that the results in Corollary \ref{thm: lower bound} can be seen as special cases of the work in \cite{ops}, where Ou et al. first proved the lower bound for commutators with respect to certain BMO spaces (using the ideas from \cite{fl,lppw}) and then
obtained the weak factorization for the predual of their BMO space in the form $H^1_{Re}(\R^{(d_1,d_2)})\otimes L^1(\R^{d_3}) +  L^1(\R^{d_1})\otimes H^1_{Re}(\R^{(d_2,d_3)})$ by duality. For more details, we refer to Section 6 in \cite{ops}.

\section{Weak factorization of the product Hardy space $h^1(\RR)$}

In this section we will first sketch the proof of Theorem \ref{thm fs}, and then we
 provide the proof of the  weak factorization for the predual of bmo$(\RR)$ characterised by
rectangular atoms (as in Definition \ref{def little h1}). We adapt the idea from \cite{u} (see also a recent refinement of the idea in \cite{dlwy}) to our current product setting for atoms and for the bmo defined via rectangles. The main approach here is to approximate each $h^1(\RR)$ atom $a$
by a related bilinear form $\Pi(f,g)$ with two $L^2(\R^2)$ functions $f$ and $g$ constructed with respect to $a$.

\begin{proof}[\bf Sketch of the proof of Theorem \ref{thm fs}]
We first consider the definition of $h^1(\rlz)$ via $q$-atoms.
Suppose $q\in(1,\infty]$. A $q$-atom on $\rlz$ is a function
$a\in L^q(\R^2)$ supported on a rectangle $R \subset \rlz$ with $\|a\|_{L^q(\R^2)} \leq |R|^{{1\over q}-1}$ and satisfying the cancellation property
$$  \int_{\R\times\R} a(x_1,x_2) dx_1dx_2=0. $$
Let $Atom_q(\rlz)$ denote the collection of all such atoms.
The atomic Hardy space $h^{1,q}(\rlz)$ is defined as the set of functions of the form
\begin{align}\label{atomic decomposition}
f=\sum_i\alpha_i a_i
\end{align}
with $\{a_i\}_i\subset Atom_q(\rlz)$, $\{\alpha_i\}_i \subset \mathbb{C}$ and $\sum_i|\alpha_i| <\infty$.
Moreover, $h^{1,q}(\rlz)$  is equipped with the norm $\|f\|_{h^{1,q}(\rlz)}:= \inf \sum_i|\alpha_i|$ where
the infimum is taken over all possible decompositions of $f$ in the form \eqref{atomic decomposition}.

Next, it suffices to prove that for $q\in(1, \fz)$,
the spaces $h^{1,q}(\rlz)$ and $h^{1,\,\fz}(\rlz)$
coincide with equivalent norms. Assuming that this is true at the moment,
then to prove the duality of $h^1(\rlz)$ with bmo$(\rlz)$, we just need to show
the dual space of $h^{1,2}(\rlz)$ is bmo$(\rlz)$. This follows from a standard argument, see for example \cite{CW2}, also \cite[Section II, Chapter 3]{J}. 

Concerning the equivalence of  the spaces $h^{1,q}(\rlz)$ and $h^{1,\,\fz}(\rlz)$, we first point out that
 the inclusion $h^{1,\,\fz}(\rlz)\subset h^{1,q}(\rlz)$ for
 $q\in(1, \fz)$ is obvious, since an $\infty$-atom must be a $q$-atom for all $q\in(1,\infty)$. Thus, we only need to establish the
converse. We do so by showing that
any $(1,q)$-atom $a$ with supp$(a)\subset R_0$,
$b:=|R_0|a$  has an atomic decomposition
$b=\sum_{i=0}^\fz \alpha_ib_i$, where each $b_i$, $i\in\mathbb{Z}_+$,
is a $(1,\fz)$-atom and
$\sum_{i=0}^\fz|\alpha_i|\ls 1.$
Actually, this follows from a standard induction argument (see for example \cite{CW2}) using the Whitney covering lemma and a variant of the argument in \cite[Lemma (3.9)]{CW2}.
\end{proof}

\begin{theorem} \label{thm:recfac} Let $\epsilon$ be an arbitrary positive number.  Let $a(x_1,x_2)$ be
an atom as defined in Definition \ref{def atom}.
Then there exist $f,g \in \ltp$  and a constant $C(\epsilon)$ depending only on $\epsilon$ such that
\[ \| a  - \Pi(f,g) \|_{h^1(\mathbb{R} \times \mathbb{R})} < \epsilon, \]
where $\| f \|_{\ltp} \|g \|_{\ltp} \le C(\epsilon)$.
\end{theorem}

To prove Theorem \ref{thm:recfac}, we first provide a technical lemma as follows.

\begin{lemma} \label{lem:H1Est}
Let $R=I\times J$ be a rectangle in $\RR$ with center $(x_I,x_J)$. For every $\epsilon>0$, we choose
$M$ such that
\begin{align}\label{M}
\frac{\log M}{ M} <\ez.
\end{align}
Then define $\widetilde{R} = \widetilde{I}\times \widetilde{J}$ as another rectangle in $\RR$ center at $(x_{\widetilde{I}},x_{\widetilde{J}})$ and satisfy: $\ell(\widetilde{I})=\ell(I)$,  $\ell(\widetilde{J})=\ell(J)$ and
$|x_I-x_{\widetilde{I}}|=M\ell(I)$, $|x_J-x_{\widetilde{J}}|=M\ell(J)$.

Let $f: \mathbb{R}^2 \rightarrow \mathbb{C}$ and assume that
 $\textnormal{supp}\, f \subseteq R\cup \widetilde{R}.$
  Further, assume that
  $$ |f(x_1,x_2)| \ls  \frac1{M|R|} \Big(\chi_R(x_1,x_2)+\chi_{\widetilde{R}}(x_1,x_2)\Big) $$
  and that $f$ has mean zero property:
\begin{align}\label{mean zero var fir}
\int_{\RR} f(x_1,x_2) \,dx_1dx_2 = 0.
\end{align}
Then $\| f \|_{h^1(\mathbb{R} \times \mathbb{R})} \lesssim \epsilon,$
where the implicit constant is independent of $f$, $\epsilon$ and $M$.
\end{lemma}

\begin{proof}
Suppose $f$ satisfies the conditions as stated in the lemma above. We will show that
$f$ has an atomic decomposition as the form in Definition \ref{def little h1}.  To see this, we first define two functions $f_1(x) $ and $f_2(x)$ by
\begin{align*}
&f_1(x_1,x_2)=f(x_1,x_2), (x_1,x_2)\in R; \quad f_1(x_1,x_2)=0, (x_1,x_2)\in \R^2\setminus R, \quad{\rm and}\\
& f_2(x_1,x_2)=f(x_1,x_2), (x_1,x_2)\in \widetilde{R}; \quad f_2(x_1,x_2)=0, (x_1,x_2)\in \R^2\setminus \widetilde{R}.
\end{align*}
Then we have $f=f_1+f_2$ and
$$ |f_1(x_1,x_2)| \ls  \frac1{M|R|} \chi_R(x_1,x_2)\quad {\rm and}\quad |f_2(x_1,x_2)| \ls  \frac1{M|R|}  \chi_{\widetilde{R}}(x_1,x_2). $$

Define
\begin{align*}
g_1^1(x_1,x_2)&:=\frac{\chi_{2R}(x_1,x_2)}{|2R|}\iint_{R}f_1(y_1,y_2)dy_1dy_2,\\
f_1^1(x_1,x_2)&:= f_1(x_1,x_2)- g_1^1(x_1,x_2),\\
\alpha_1^1&:=\|f_1^1\|_\infty |2R|.
\end{align*}
Then we claim that $a_1^1:= (\alpha_1^1)^{-1} f_1^1$  is a rectangle atom as in Definition \ref{def atom}.
First, it is direct that $a_1^1$ is supported in $2R$. Moreover, we have that
\begin{align*}
  \int_{\R^2} a_1^1(x_1,x_2) \,dx_1dx_2& =(\alpha_1^1)^{-1}\int_{\R^2} \left(f_1(x_1,x_2)- g_1^1(x_1,x_2) \right)dx_1dx_2\\
  &= (\alpha_1^1)^{-1}\bigg(\int_{\R^2} f_1(x_1,x_2)dx_1dx_2- \int_{\R^2} f_1(x_1,x_2)dx_1dx_2\bigg)\\
  &=0
\end{align*}
and that
\begin{align*}
\|a_1^1\|_\infty \leq |(\alpha_1^1)^{-1}|\|f_1^1\|_\infty = \frac1{|2R|}.
\end{align*}
Thus, $a_1^1$  is an atom as in Definition \ref{def atom}. Moreover, we have
\begin{align*}
|\alpha_1^1|=\|f_1^1\|_\infty |2R| \le \|f_1\|_\infty |2R| \ls \frac{1}{M |R|} \cdot |2R|\ls \frac{1}{M}.
\end{align*}
And
\begin{align*}
f_1(x_1,x_2)= f_1^1(x_1,x_2)+ g_1^1(x_1,x_2)= \alpha_1^1 a_1^1+ g_1^1(x_1,x_2).
\end{align*}
For $g_1^1(x_1,x_2)$, we further write it as
\begin{align*}
g_1^1(x_1,x_2)= g_1^1(x_1,x_2)- g_1^2(x_1,x_2)+ g_1^2(x_1,x_2)=:f_1^2(x_1,x_2)+g_1^2(x_1,x_2)
\end{align*}
with
$$g_1^2(x_1,x_2):=\frac{\chi_{4R}(x_1,x_2)}{|4R|}\iint_{R}f_1(y_1,y_2)dy_1dy_2. $$
Again, we define
\begin{align*}
\alpha_1^2&:=\|f_1^2\|_\infty |4R| \quad{\rm and}\quad a_1^2:= (\alpha_1^2)^{-1} f_1^2,
\end{align*}
and following similar estimates as for $a_1^1$, we see that
$a_1^2$  is an atom as in Definition \ref{def atom} with
$$ \|a_1^2\|_\infty \leq \frac{1}{|4R|}\quad \textnormal{ and } \quad \left\vert \alpha_1^2\right\vert \lesssim \frac{1}{M}. $$
Then we have
\begin{align*}
f_1(x_1,x_2)=\sum_{i=1}^2 \alpha_1^i a_1^i+ g_1^2(x_1,x_2).
\end{align*}
Continuing in this fashion we see that for $i \in \{1, 2, . . . , i_0\}$,
\begin{align*}
f_1(x_1,x_2)=\sum_{i=1}^{i_0} \alpha_1^i a_1^i+ g_1^{i_0}(x_1,x_2),
\end{align*}
where for $i \in \{2, . . . , i_0\}$,
\begin{align*}
g_1^{i}(x_1,x_2)&:=\frac{\chi_{2^{i}R}(x_1,x_2)}{|2^i R|}\iint_{R}f_1(y_1,y_2)dy_1dy_2,\\
f_1^{i}(x_1,x_2)&:= g_1^{i-1}(x_1,x_2)-g_1^{i}(x_1,x_2),\\
\alpha_1^i&:=\|f_1^i\|_\infty |2^iR| \quad{\rm and}\\
a_1^i&:= (\alpha_1^i)^{-1} f_1^i.
\end{align*}
Here we choose $i_0$ to be the smallest positive integer such that $\widetilde{R} \subset 2^{i_0}R$. Then from the
definition of $\widetilde{R}$, we obtain that
$$ i_0\approx \log_2 M. $$

Moreover, for $i \in \{1, 2, . . . , i_0\}$, we have
\begin{align*}
|\alpha_1^i|&\leq \|f_1^i\|_\infty |2^iR| \leq |2^iR| \big(\|g_1^{i-1}\|_\infty +\|g_1^{i}\|_\infty\big)\\
&\leq |2^iR|  \bigg(\frac{1}{|2^{i-1} R|}\iint_{R}|f_1(y_1,y_2)|dy_1dy_2 +{1\over |2^i R|}\iint_{R}|f_1(y_1,y_2)|dy_1dy_2\bigg)\\
&\ls |2^iR| \frac{1}{|2^{i-1} R|} \|f_1\|_\infty |R| \\
&\ls |R|\frac{1}{ M|R|}\\
&=\frac{1}{M}.
\end{align*}

Following the same steps, we also obtain that  for $i \in \{1, 2, . . . , i_0\}$,
\begin{align*}
f_2(x_1,x_2)=\sum_{i=1}^{i_0} \alpha_2^i a_2^i+ g_2^{i_0}(x_1,x_2),
\end{align*}
where for $i \in \{2, . . . , i_0\}$,
\begin{align*}
g_2^{i}(x_1,x_2)&:={\chi_{2^{i}R}(x_1,x_2)\over |2^i R|}\iint_{\widetilde R}f_2(y_1,y_2)dy_1dy_2,\\
f_2^{i}(x_1,x_2)&:= g_2^{i-1}(x_1,x_2)-g_2^{i}(x_1,x_2),\\
\alpha_2^i&:=\|f_2^i\|_\infty |2^iR| \quad{\rm and}\\
a_2^i&:= (\alpha_2^i)^{-1} f_2^i.
\end{align*}
Similarly, for $i \in \{1, 2, . . . , i_0\}$, we have
\begin{align*}
|\alpha_2^i|\ls {1\over M}.
\end{align*}

Combining the decompositions above, we obtain that
\begin{align*}
f(x_1,x_2)=\sum_{j=1}^2\sum_{i=1}^{i_0} \alpha_j^i a_j^i+ g_j^{i_0}(x_1,x_2).
\end{align*}
We now consider the tail $g_1^{i_0}(x_1,x_2) + g_2^{i_0}(x_1,x_2)$. To handle that, consider
the rectangle $\overline R$  centered at the point
$$ \Big( {x_I+x_{\widetilde{I}}\over 2}, {x_J+x_{\widetilde{J}}\over 2} \Big) $$
with sidelength $2^{i_0+1}\ell(I)$ and $2^{i_0+1}\ell(J)$. Then, it is clear that $R\cup \widetilde{R} \subset \overline R$,
and that $2^{i_0}R, 2^{i_0}\widetilde{R} \subset \overline R$.
Thus, we get that
\begin{align*}
{\chi_{\overline{R}}(x_1,x_2)\over |\overline{R}|}\iint_{\overline{R}}f_1(y_1,y_2)dy_1dy_2+
\ {\chi_{\overline{R}}(x_1,x_2)\over |\overline{R}|}\iint_{\overline{R}}f_2(y_1,y_2)dy_1dy_2=0.
\end{align*}
Hence, we write
\begin{align*}
g_1^{i_0}(x_1,x_2) + g_2^{i_0}(x_1,x_2)&=\bigg(g_1^{i_0}(x_1,x_2) -
{\chi_{\overline{R}}(x_1,x_2)\over |\overline{R}|}\iint_{\overline{R}}f_1(y_1,y_2)dy_1dy_2\bigg) \\
&\quad+ \bigg( g_2^{i_0}(x_1,x_2) -
 {\chi_{\overline{R}}(x_1,x_2)\over |\overline{R}|}\iint_{\overline{R}}f_2(y_1,y_2)dy_1dy_2\bigg)\\
 &=: f_1^{i_0+1}+f_2^{i_0+1}.
\end{align*}
For $j=1,2$, we now define
\begin{align*}
\alpha_j^{i_0+1}&:=\|f_j^{i_0+1}\|_\infty |2^{i_0+1}R| \quad{\rm and}\\
a_j^{i_0+1}&:= (\alpha_j^{i_0+1})^{-1} f_j^{i_0+1}.
\end{align*}
Again we can verify that for $j=1,2$, $a_j^{i_0+1}$ is an atom as in Definition \ref{def atom} with
$$\|a_j^{i_0+1}\|_\infty = {1\over  |2^{i_0+1}R| }.$$
Moreover, we also have
$$|\alpha_j^{i_0+1}| \ls {1\over M }.$$

Thus, we obtain that
\begin{align*}
f(x_1,x_2)=\sum_{j=1}^2\sum_{i=1}^{i_0+1} \alpha_j^i a_j^i,
\end{align*}
which implies that $f\in h^1(\RR)$ and
\begin{align*}
\|f\|_{h^1(\RR)} &\leq \sum_{j=1}^2\sum_{i=1}^{i_0+1}  |\alpha_j^i| \ls  \sum_{j=1}^2\sum_{i=1}^{i_0+1}  {1\over M}
\ls {\log M\over M}< \epsilon.
\end{align*}
Therefore, we finish the proof of Lemma \ref{lem:H1Est}.
\end{proof}

\bigskip
\begin{proof}[\bf Proof of Theorem \ref{thm:recfac}]  Suppose $a$ is an atom of $h^1(\RR)$ supported in a rectangle
$R$ centered at $(x_I, x_J)$, as in Definition \ref{def atom}.
For every fixed $\epsilon>0$, we
now let $M$, $\widetilde{R}$
be the same as in Lemma \ref{lem:H1Est}.

We define the two functions
\[ f(x_1,x_2): = \textbf{1}_{\widetilde{R}}(x_1,x_2) \ \text{ and } g(x) := \frac{a(x_1,x_2)}{H_1H_2f(x_I,x_J)}.\]
Then by definition, we have
\[ \|f \|_{\ltp} = |\widetilde{R}|^{\frac{1}{2}}= |R|^{\frac{1}{2}} \]
and
\[ \| g \|_{\ltp} = \frac{1}{ |   H_1H_2f(x_I,x_J) |} \| a \|_{\ltp} \le \frac{|R|^{-\frac{1}{2}}}{ |   H_1H_2f(x_I,x_J) |}.\]
Observe that
\[ |H_1 H_2 f(x_I,x_J) |= \left | \int_{\widetilde{R}} \frac{1}{x_I-y_1}\frac{1}{x_J-y_2 } dy_1 dy_2 \right| \approx \frac{1}{M^2}.\]
Thus, we have that
\[ \|f \|_{\ltp}  \| g \|_{\ltp} \leq CM^2 \]
with the positive constant $C$ independent of $a(x_1,x_2)$ and $M$.
We take $C(\epsilon)$  as
\begin{equation}\label{c-epsilon}
 C(\epsilon) := CM^2,
 \end{equation}
then it is easy to see that $C(\epsilon)$ depends only on $\epsilon$ as $M$ only depends on $\epsilon$.  Now, write
\begin{align*}
 a - \Pi(f,g)& = \left( a - g H_1H_2 f \right)  +f H_1H_2 g =:w_1(x) +w_2(x).
 \end{align*}

First, consider $w_1.$ Observe that $\text{supp } w_1 \subseteq R$ and
\[ |w_1(x_1,x_2) |  = \frac{|a(x_1,x_2)|}{ |H_1H_2f(x_I,x_J)|} \left| H_1H_2f(x_I,x_J) - H_1H_2f(x_1,x_2) \right|.\]
Then as $x \in R$,  we can estimate
\begin{align*}
 &\left| H_1H_2f(x_I,x_J) - H_1H_2f(x_1,x_2) \right|\\
  &\quad= \left|
 \int_{\widetilde{R}} \frac{1}{(y_1-x_I)(y_2-x_J)} - \frac{1}{(y_1-x_1)(y_2-x_2)} dy_1dy_2 \right| \\
&\quad \leq   \int_{\widetilde{R}} \frac{|x_1-x_I|}
{|y_1-x_I||y_1-x_1||y_2-x_J|}  + \frac{|x_2-x_J|}
{|y_1-x_1||y_2-x_2||y_2-x_J|} \dyz  \\
&\quad \leq   \int_{\widetilde{R}} \frac{\ell(I)}
{M^2\ell(I)^2 M\ell(J)}  + \frac{\ell(J)}
{M\ell(I) M^2\ell(J)^2} \dyz \\
&\quad\lesssim \frac{1}{M^3}.
\end{align*}
Combining this with the definition of $w_1$ immediately gives:
\begin{align*}
|w_1(x_1,x_2)|\ls {1\over M} |a(x_1,x_2)|,
\end{align*}
which implies that
\[ \| w_1\|_{\ltp} \ls \frac{1}{M}  \| a \|_{\ltp} \ls\frac1{M|R|^\frac12}.\]

Now, consider $w_2(x_1,x_2)$.
Note that
\[ w_2(x_1,x_2) =f (x_1,x_2)H_1H_2 g(x_1,x_2)  = \frac{1}{ H_1H_2f(x_I,x_J)} \textbf{1}_{\widetilde{R}} (x_1,x_2) H_1H_2 a(x_1,x_2). \]
Clearly, $\text{supp } w_2 \subseteq \widetilde{R}.$ Furthermore, using the mean zero property of $a(x_1,x_2)$, we have:
\begin{align*}
H_1H_2a(x_1,x_2) &= \int_{R} \frac{a(y_1,y_2)}{(y_1-x_1)(y_2-x_2)} dy_1dy_2 \\
&=\int_{R} \bigg( \frac{1}{(y_1-x_1)(y_2-x_2)} - \frac{1}{(x_I-x_1)(x_J-x_2)} \bigg) a(y_1,y_2) dy_1dy_2.
\end{align*}
It is immediate that
\begin{align*}
|  \textbf{1}_{\widetilde{R}} (x_1,x_2) H_1H_2a(x_1,x_2) | &\ls  \textbf{1}_{\widetilde{R}}(x_1,x_2) {1\over M^3} \|a\|_{L^\infty}.
 \end{align*}
Thus, we can conclude that
\begin{align*}
|w_2(x_1,x_2)|
&\ls \textbf{1}_{\widetilde{R}}(x_1,x_2) {1\over M} \|a\|_{L^\infty},
 \end{align*}
which implies that
\[ \|w_2 \|_{\ltp} \ls \frac1{M|R|^\frac12}.
\]

Combining the estimates of $w_1$ and $w_2$, we can conclude that $a - \Pi(f,g)$ has support contained in
$$R\cup \widetilde{R}$$ and satisfies
\[ \| a - \Pi(f,g) \|_{\ltp} \lesssim  \frac1{M|R|^\frac12}. \]
Moreover, from the definition of the bilinear form, we obtain that
$$ \int_{\R^2}\left( a(x_1,x_2) -  \Pi(f,g) (x_1,x_2) \right)dx_1dx_2=0.$$
Then, the fact that  $\| a - \Pi(f,g) \|_\hop \ls \epsilon$ now immediately follows from Lemma \ref{lem:H1Est}.
\end{proof}

\begin{remark}\label{remark compact}
From the proof of Theorem  \ref{thm:recfac} as above, we observe that the functions $f$ and $g$ that we constructed
are actually in $L^\infty(\R^2)$ with compact support.
\end{remark}

Now we provide the proof of the main result in this  paper. To begin with, we need the following two auxiliary lemmas.
\begin{lemma}\label{lemma aux1}
Suppose $b\in ${\rm bmo}$(\RR)$. Then we have
\begin{align}\label{aux1}
\big\| [  b, H_1H_2 ] \big\|_{L^2(\R^2)\to L^2(\R^2)} \ls \|b\|_{{\rm bmo}(\RR)},
\end{align}
where the implicit constant is independent of $b$.
\end{lemma}

\begin{proof}

We point out that the proof of upper bound of $[b,H_1H_2]$ follows directly from the property of bmo$(\RR)$ and the $L^2$ boundedness of the Hilbert transforms $H_1$ and $H_2$.

Suppose that $b\in $bmo$(\RR)$. Then we know that
for any fixed $x_2\in\R$, $b(x_1,x_2)$ as a function of $x_1$ is in the standard one-parameter BMO$(\R)$, symmetric result holds for the roles of $x_1$ and $x_2$ interchanged. Moreover, we further have that
\begin{align}\label{bmo norm equiv}
\|b\|_{\rm bmo(\RR)} \approx \sup_{x_1\in\R} \|b(x_1,\cdot)\|_{\rm BMO(\R)} +
 \sup_{x_2\in\R} \|b(\cdot,x_2)\|_{\rm BMO(\R)},
\end{align}
where the implicit constants are independent of the function $b$.

Next, we point out that
 $$[b, H_1H_2]= H_1[b,H_2] + [b,H_1]H_2.$$
 Then based on \eqref{bmo norm equiv} and the result of Coifman--Rochberg--Weiss \cite{crw}, we know that
\begin{align*}
&\big\|[b,H_2] \big\|_{L^2(\R^2)\to L^2(\R^2)} + \big\| [b,H_1]\big\|_{L^2(\R^2)\to L^2(\R^2)} \\
&\quad\ls \sup_{x_1\in\R} \|b(x_1,\cdot)\|_{\rm BMO(\R)} +
 \sup_{x_2\in\R} \|b(\cdot,x_2)\|_{\rm BMO(\R)}\\
&\quad\ls \|b\|_{{\rm bmo}(\RR)}.
\end{align*}

Then, denote by ${\rm Id}_1$ and ${\rm Id}_2$ the identity operator on $L^2(\R)$ for the first and second variable, respectively. We
further have
$$[b, H_1H_2]= (H_1\otimes {\rm Id}_2) \circ   [b,H_2] + [b,H_1] \circ ({\rm Id}_1\otimes H_2),$$
where we use $T_1 \circ T_2$ to denote the composition of two operators $T_1$ and $T_2$. Thus, we obtain that
\begin{align*}
&\big\| [  b, H_1H_2 ] \big\|_{L^2(\R^2)\to L^2(\R^2)} \\
&\quad=\big\| (H_1\otimes {\rm Id}_2) \circ   [b,H_2] + [b,H_1] \circ ({\rm Id}_1\otimes H_2) \big\|_{L^2(\R^2)\to L^2(\R^2)}\nonumber\\
&\quad \leq \big\| (H_1\otimes {\rm Id}_2) \circ  [b,H_2]  \big\|_{L^2(\R^2)\to L^2(\R^2)}  + \big\|[b,H_1]\circ ({\rm Id}_1\otimes H_2)  \big\|_{L^2(\R^2)\to L^2(\R^2)}\nonumber\\
& \quad\leq   \big\| H_1\|_{L^2(\R^2)\to L^2(\R^2)} \big\|[b,H_2] \big\|_{L^2(\R^2)\to L^2(\R^2)} \nonumber\\
&\quad\quad+ \big\| [b,H_1]\big\|_{L^2(\R^2)\to L^2(\R^2)}\big\|H_2 \big\|_{L^2(\R^2)\to L^2(\R^2)}\nonumber\\
&\quad\ls \|b\|_{{\rm bmo}(\RR)},\nonumber
\end{align*}
which shows that \eqref{aux1} holds.
\end{proof}


\begin{lemma}\label{lemma aux2}
Suppose $b\in ${\rm bmo}$(\RR)$, and $f,g \in L^\infty(\R^2)$ with compact supports.
Then the bilinear form $\Pi(f,g)$ defined as in \eqref{bilinear form} is in $h^1(\RR)$ with the norm satisfying
\begin{align}\label{aux2}
\| \Pi(f,g) \|_{h^1(\mathbb{R} \times \mathbb{R})} \ls \|f\|_{L^2(\R^2)} \| g \|_{L^2(\R^2)},
\end{align}
where the implicit constant is independent of $f$ and $g$.
\end{lemma}
\begin{proof}

We first note that for every $b\in $bmo$(\RR)$, $b$ is in $L^q_{loc}(\R^2)$ for $q\in(1,\infty)$. In fact,
for any compact set $\Omega$ in $\RR$, there exist two closed intervals $I,J\in\mathbb R$, such that
$\Omega\subset I\times J$. For any $x_1\in I$, we have $b(x_1,x_2)$ as a function of $x_2$ is in BMO$(\R)$.
Hence, $b(x_1,x_2)$ as a function of $x_2$ is in $L^q(J)$. Again, for any $x_2\in J$, $b(x_1,x_2)$ as a function of $x_1$ is in $L^q(I)$. As a consequence, we have that for any $q\in(1,\infty)$,
\begin{align*}
\int_{\Omega} |b(x_1,x_2)|^q dx_1dx_2 &\leq \int_{I}\int_{ J} |b(x_1,x_2)|^q dx_2dx_1
\leq  \int_{I}  \|b(x_1,\cdot)\|_{{\rm BMO}(\R)}^q  dx_1\\
&\leq \sup_{x_1\in I}\|b(x_1,\cdot)\|_{{\rm BMO}(\R)}^q \ |I|\\
&\leq C\|b\|^q_{\rm bmo(\RR)}  \ |I|,
\end{align*}
which shows that $b$ is in $L^q_{loc}(\R^2)$ for $q\in(1,\infty)$ with
\begin{align}\label{loc L2}
\int_{\Omega} |b(x_1,x_2)|^q dx_1dx_2
\leq C_\Omega\|b\|_{\rm bmo(\RR)}^q
\end{align}
for any compact set $\Omega \in \RR$.

We now consider the property of the bilinear form $\Pi(f,g)$ defined as in \eqref{bilinear form}.
For each $f,g \in L^\infty(\R^2)$ with compact support, we have that
$\Pi(f,g) = g H_1H_2 f - f H_1H_2 g$  is in $L^2(\R^2)$ with compact support.
In fact, since $f$ is in $L^\infty(\R^2)$ with compact support, we get that
$f$ is in $L^2(\R^2)$ with compact support, which implies that
 $H_1H_2 f $ is in $L^2(\R^2)$, and hence
$ g H_1H_2 f  $ is in $L^2(\R^2)$ with compact support. Similar argument holds for $f H_1H_2 g$.
Also note that from \eqref{loc L2}, for each $b\in {\rm bmo}(\RR)$, $b$ is in $L^2_{loc}(\R^2)$.
We have that
$$
\big\vert \left \langle b, \Pi(f,g)\right \rangle_{L^2(\R^2)} \big\vert =\Big| \int_{\RR}b(x_1,x_2) \Pi(f,g)(x_1,x_2)dx_1dx_2 \Big| \leq C\|b\|_{\rm bmo(\RR)}<\infty,
$$
where the constant $C$ depends on the support of $f$ and $g$. Hence $\left \langle b, \Pi(f,g)\right \rangle_{L^2(\R^2)}$ is well-defined.

Next we claim that for each $f,g \in L^\infty(\R^2)$ with compact support,
\begin{align}\label{claim ee}
 \left \langle b, \Pi(f,g)\right \rangle_{L^2(\R^2)}
& =  \left \langle  \left[  b, H_1 H_2 \right] f ,g \right \rangle_{L^2(\R^2)}.
\end{align}
To see this, note that by definition of $\Pi(f,g)$,
\begin{align*}
 \left \langle b, \Pi(f,g)\right \rangle_{L^2(\R^2)}
& =  \left \langle b, g H_1H_2 f - f H_1H_2 g \right \rangle_{L^2(\R^2)}.
\end{align*}
Next, since $f,g \in L^\infty(\R^2)$ with compact support and $b\in L^2_{loc}(\R^2)$, it is direct that
\begin{align*}
 \left \langle b, g H_1H_2 f \right \rangle_{L^2(\R^2)} =  \left \langle g, b H_1H_2 f \right \rangle_{L^2(\R^2)}
\end{align*}
and that
\begin{align*}
  \left \langle b,  f H_1H_2 g \right \rangle_{L^2(\R^2)} &= \int_{\RR} b(x_1,x_2)f(x_1,x_2) H_1H_2 g(x_1,x_2)dx_1dx_2 \\
  &= \int_{\RR} H_1H_2 (b\cdot f)(x_1,x_2)  g(x_1,x_2)dx_1dx_2 \\
  &= \left \langle H_1H_2(b \cdot f) , g \right \rangle_{L^2(\R^2)}.
\end{align*}
Combining these two equalities, we get that the claim \eqref{claim ee} holds.

From the claim \eqref{claim ee} and the upper bound as in \eqref{aux1}, we obtain that
\begin{align}\label{predual}
\left\vert \left \langle b, \Pi(f,g)\right \rangle_{L^2(\R^2)} \right\vert
& = \left\vert \left \langle  \left[  b, H_1 H_2 \right] f ,g \right \rangle_{L^2(\R^2)}\right\vert
 \lesssim \|b \|_{{\rm bmo}(\RR)} \|f \|_{L^2(\mathbb{R}^2)} \|g \|_{L^2(\mathbb{R}^2)},
\end{align}
where the implicit constant is independent of $f$ and $g$.

Now for any fixed $f,g\in L^\infty(\R^2)$ with compact support, we claim that $\Pi(f,g)$ is in $h^1(\RR)$.

To see this, we now show that $\Pi(f,g)$ is the product of a constant and a $2$-atom of $h^1(\RR)$.
In fact,  from the definition of the bilinear form, we obtain that
$$ \int_{\R^2} \Pi(f,g) (x_1,x_2) dx_1dx_2=0.$$
Next,  since both $f$ and $g$ are in $L^\infty(\R^2)$ with compact support,
we get that $\Pi(f,g) \in L^2(\RR)$ with compact support, denoted it by a rectangle $R\subset \RR$. And we further have 
$ \|\Pi(f,g)\|_{L^2(\R^2)}\leq C_{f,g} \|g\|_{L^\infty(\R^2)} \|f\|_{L^\infty(\R^2)} $, where the constant $C_{f,g}$ depends on the compact supports of $f$ and $g$. Moreover, we assume that $\|\Pi(f,g)\|_{L^2(\R^2)}\not=0$ since otherwise $\Pi(f,g)=0$ almost everywhere and hence it is in  $h^1(\RR)$.

Now we can write 
$$ \Pi(f,g)(x_1,x_2) 
=: a(x_1,x_2) \cdot  \|\Pi(f,g)\|_{L^2(\R^2)} |R|^{1\over2},$$
where
$$ a(x_1,x_2) := {\Pi(f,g)(x_1,x_2)\over \|\Pi(f,g)\|_{L^2(\R^2)} |R|^{1\over2} }.$$
Then it is direct that $a(x_1,x_2) $ is supported in $R$, $\int_{\RR} a(x_1,x_2) dx_1dx_2=0$ and that
$ \|a\|_{L^2(\R^2)}\leq |R|^{-{1\over2}} $. Hence $a(x_1,x_2)$ is a $2$-atom of $h^1(\RR)$, which implies that 
$\Pi(f,g)$ is in $h^1(\RR)$, i.e., the claim holds.

Note that $\Pi(f,g)$ is in $h^1(\RR)$, we then further have
\begin{align*}
\|h\|_{h^1(\mathbb{R} \times \mathbb{R})} \approx \sup_{\|b\|_{\rm bmo (\mathbb R\times\mathbb R)} \leq 1} \big|\langle b,h \rangle\big|,
\end{align*}
which follows from the fundamental fact as in 1.4.12 (b) in \cite{Gra}.

This, together with \eqref{predual}, immediately implies that \eqref{aux2} holds.
\end{proof}

We now provide the proof of our main result.
\begin{proof}[\bf Proof of Theorem \ref{thm: main}]
We first point out from Remark \ref{remark main compact}, the functions $g_j^k$ and $h^k_j$ in the representation \eqref{represent of H1}
are actually in $L^\infty(\R^2)$ with compact support.
Hence, from \eqref{aux2},
%
for every $f\in\hop$  having the representation \eqref{represent of H1} with
$$\sum_{k=1}^\fz\sum_{j=1}^\fz\lf|\az^k_j\r|\lf\|g^k_j\r\|_\ltp\lf\|h^k_j\r\|_\ltp<\fz,$$
 it follows that
\begin{eqnarray*}
\|f\|_\hop&\ls&\inf\lf\{\sum_{k=1}^\fz\sum_{j=1}^\fz\lf|\az^k_j\r|\lf\|g^k_j\r\|_\ltp\lf\|h^k_j\r\|_\ltp:
f=\sum_{k=1}^\fz\sum_{j=1}^\fz \az^k_j\,\Pi\lf(g^k_j,h^k_j\r)\r\}.
\end{eqnarray*}

It remains to show that for each $f\in\hop$, $f$ has a representation as in \eqref{represent of H1} with
\begin{equation}\label{lower bound of H1 respent}
\inf\lf\{\sum_{k=1}^\fz\sum_{j=1}^\fz\lf|\az^k_j\r|\lf\|g^k_j\r\|_\ltp\lf\|h^k_j\r\|_\ltp:
f=\sum_{k=1}^\fz\sum_{j=1}^\fz \az^k_j\,\Pi\lf(g^k_j,h^k_j\r)\r\}\ls \|f\|_\hop.
\end{equation}
To this end, assume that $f$ has the following atomic representation
 $\displaystyle f=\sum_{j=1}^\fz\az^1_ja^1_j$ with $\displaystyle\sum_{j=1}^\fz|\az^1_j|\le C_0\|f\|_\hop$
 for certain absolute constant $C_0\in(1, \fz)$.
We show that for every $\epsilon\in\left(0, C_0^{-1}\right)$ and every $K\in\nn$, $f$ has the following representation
\begin{equation}\label{itration}
f=\sum_{k=1}^K\sum_{j=1}^\fz\az^k_j\,\Pi\lf(g^k_j, h^k_j\r)+E_K,
\end{equation}
where
\begin{equation}\label{itration-2}
\sum_{j=1}^\fz\lf|\az^k_j\r|\le  \ez^{k-1}C_0^k\|f\|_\hop,
\end{equation}
and $E_K\in \hop$ with
\begin{equation}\label{itration-3}
\|E_K\|_\hop\le (\ez C_0)^K\|f\|_\hop,
\end{equation}
and $g^k_j\in\ltp$, $h^k_j\in\ltp$ for each $k$ and $j$,
$\{\az^k_j\}_{j}\in \ell^1$ for each $k$ satisfying that
\begin{equation}\label{itration-1}
\lf\|g^k_j\r\|_\ltp\lf\|h^k_j\r\|_\ltp\ls C(\ez)
\end{equation}
with the absolute constant $C(\ez)$ defined as in \eqref{c-epsilon}.

In fact, for given $\ez$ and each $a^1_j$,  by Theorem \ref{thm:recfac} we obtain that
 there exist
 $g^1_j\in\ltp$ and $h^1_j\in\ltp$ with
 $$\lf\|g^1_j\r\|_\ltp\lf\|h^1_j\r\|_\ltp\ls C(\ez)$$
  and
 $$\lf\| a^1_j-\Pi\lf(g^1_j,h^1_j\r)\r\|_\hop<\ez.$$
Actually, from Remark \ref{remark compact}, these two functions $g^1_j$ and $h^1_j$
are in $L^\infty(\R^2)$ with compact supports.

Now we write
\begin{align*}
f&=\sum_{j=1}^\fz\az^1_ja^1_j
=\sum_{j=1}^\fz\az^1_j\Pi\lf(g^1_j,h^1_j\r)+ \sum_{j=1}^\fz\az^1_j\lf[ a^1_j-\Pi\lf(g^1_j,h^1_j\r)\r]\\
&=:M_1+E_1.
\end{align*}
Observe that
\begin{align*}
\|E_1\|_\hop&\le \sum_{j=1}^\fz\lf|\az^1_j\r| \lf\| a^1_j-\Pi\lf(g^1_j,h^1_j\r)\r\|_\hop \le \ez C_0\|f\|_\hop.
\end{align*}

Since $E_1\in\hop$, for the given $C_0$,
there exists a sequence of atoms $\{a^2_j\}_j$ and numbers $\{\az^2_j\}_j$
such that $\displaystyle E_1=\sum_{j=1}^\fz\az^2_ja^2_j$ and
\begin{equation*}
\sum_{j=1}^\fz\lf|\az^2_j\r|\le C_0\|E_1\|_\hop\le \ez C_0^2\|f\|_\hop.
\end{equation*}
Again,  we have that for given $\ez$, there exists a representation of $E_1$ such that
\begin{align*}
E_1&=\sum_{j=1}^\fz\az^2_j\Pi\lf(g^2_j,h^2_j\r)+ \sum_{j=1}^\fz\az^2_j\lf[ a^2_j-\Pi\lf(g^2_j,h^2_j\r)\r]
\\
&=:M_2+E_2,
\end{align*}
and
\begin{equation*}
\lf\|g^2_j\r\|_\ltp\lf\|h^2_j\r\|_\ltp\ls C(\ez)\,\ {\rm and}\,\, \lf\| a^2_j-\Pi\lf(g^2_j, h^2_j\r)\r\|_\hop<\frac{\ez}{2}.
\end{equation*}
Moreover,
\begin{eqnarray*}
\|E_2\|_\hop&\le& \sum_{j=1}^\fz\lf|\az^2_j\r| \lf\| a^2_j-\Pi\lf(g^2_j,h^2_j\r)\r\|_\hop
\le (\ez C_0)^2\|f\|_\hop.
\end{eqnarray*}
Now we conclude that
\begin{eqnarray*}
f=\sum_{j=1}^\fz\az^1_ja^1_j=\sum_{k=1}^2\sum_{j=1}^\fz\az^k_j\Pi\lf(g^k_j, h^k_j\r)+E_2,
\end{eqnarray*}
Again,  from Remark \ref{remark compact}, all these functions $g^k_j$ and $h^k_j$
are in $L^\infty(\R^2)$ with compact supports.

Continuing in this way, we deduce that for every $K\in\nn$, $f$ has the representation \eqref{itration} satisfying
\eqref{itration-1}, \eqref{itration-2}, and \eqref{itration-3}. Thus letting $K\to\fz$, we
see that \eqref{represent of H1} holds. Moreover, since $\ez C_0<1$, we have that
$$\sum_{k=1}^\fz\sum_{j=1}^\fz \lf|\az^k_j\r|\le \sum_{k=1}^\fz\ez^{-1}(\ez C_0)^k\|f\|_\hop\ls \|f\|_\hop,$$
which implies \eqref{lower bound of H1 respent} and hence, completes the proof of Theorem  \ref{thm: main}.
\end{proof}

Next, by duality, we provide the proof of our second main result in this  paper.

\begin{proof}[\bf Proof of Corollary \ref{thm: lower bound}]
Suppose that $b\in \cup_{q>1}L^q_{loc}(\R^2)$. Assume that $\mbhh$ is bounded on $\ltp$  and $f\in\hop$ and $f$ has compact support.
From Theorem \ref{thm: main},
we deduce that
\begin{eqnarray*}
\langle b, f\rangle_{L^2(\mathbb{R}^2)}&=&\sum_{k=1}^\fz\sum_{j=1}^\fz \az^k_j \lf\langle b, \Pi\lf(g^k_j,h^k_j\r)\r\rangle_{L^2(\mathbb{R}^2)}
=\sum_{k=1}^\fz\sum_{j=1}^\fz \az^k_j\lf\langle g^k_j,\mbhh h^k_j\r\rangle_{L^2(\mathbb{R}^2)},
\end{eqnarray*}
where in the second equality we have applied
the fact that
$$ \lf\langle b, \Pi\lf(g^k_j,h^k_j\r)\r\rangle_{L^2(\mathbb{R}^2)} =  \lf\langle g^k_j,\mbhh h^k_j\r\rangle_{L^2(\mathbb{R}^2)},$$
which follows from \eqref{claim ee} since the functions $g^k_j,h^k_j$ here are constructed as in $L^\infty(\R^2)$ with compact support (see Remark \ref{remark compact}).

This implies that
\begin{eqnarray*}
\lf|\langle b, f\rangle_{L^2(\mathbb{R}^2)}\r|
&&\le\sum_{k=1}^\fz\sum_{j=1}^\fz \lf|\az^k_j\r|\lf\|g^k_j\r\|_\ltp\lf\|\mbhh h^k_j\r\|_\ltp\\
&&\le\lf\|\mbhh:\ltp\to\ltp \r\|\ \sum_{k=1}^\fz \sum_{j=1}^\fz \lf|\az^k_j\r|\lf\|g^k_j\r\|_\ltp\lf\|h^k_j\r\|_\ltp\\
& &\ls\lf\|\mbhh:\ltp\to\ltp\r\|\|f\|_\hop.
\end{eqnarray*}
Then by the fact that $\{f\in \hop: \,f {\rm\ has\ compact\ support} \}$ is dense in $\hop$, and
the duality between $\hop$ and $\bmop$ (see \cite{fs}), we finish the proof of Corollary \ref{thm: lower bound}.
\end{proof}

\bigskip
{\bf Acknowledgement: } The authors would like to thank the editor for his patience and thank the referee for careful reading and checking, and for all the helpful
suggestions and comments, which helps to make this paper more accurate and readable.

The second author would like to thank Professor Jill Pipher for pointing out this question during her visit to Macquarie University in March 2016.

The second and third authors would like to thank the University of Hawaii for hospitality while visiting there in April 2016.

\end{document}